\documentclass[12pt, reqno]{amsart}
\usepackage{amsmath, amsthm, amscd, amsfonts, amssymb, graphicx, color,cite}
\usepackage[bookmarksnumbered, colorlinks, plainpages]{hyperref}
\hypersetup{colorlinks=true,linkcolor=red, anchorcolor=green, citecolor=cyan, urlcolor=red, filecolor=magenta, pdftoolbar=true}

\newtheorem{thm}{Theorem}[section]
\newtheorem{cor}[thm]{Corollary}
\newtheorem{lem}[thm]{Lemma}
\newtheorem{prop}[thm]{Proposition}
\theoremstyle{definition}

\theoremstyle{remark}

\numberwithin{equation}{section}

\begin{document}
\setcounter{page}{1}

\title[Volume inequalities]{Volume inequalities for asymmetric Orlicz zonotopes}

\author[F. Chen, C. Yang,  M. Luo]{Fangwei Chen$^1$, Congli Yang$^2$,  Miao Luo$^{2}$,}

\address{1. Department of Mathematics and Statistics, Guizhou University of Finance and Economics,
Guiyang, Guizhou 550004, People's Republic of China}

\email{cfw-yy@126.com; chen.fangwei@yahoo.com}
\address{2, 3. School of Mathematics and Computer Science, Guizhou Normal
University, Guiyang, Guizhou 550001, People's Republic of China.}
\email{yangcongli@gznu.edu.cn}
\thanks{The  work is supported in part by CNSF (Grant No. 11161007, Grant No. 11101099), Guizhou (Unite) Foundation for Science and Technology (Grant
No. [2014] 2044, No. [2012] 2273, No. [2011] 16), Guizhou Technology Foundation for Selected Overseas Chinese Scholar and Doctor foundation of Guizhou Normal University.}


\subjclass[2010]{52A20, 52A40, 52A38.}

\keywords{Orlicz Minkowski sum, asymmetric Orlicz zonotopes, Shadow system, volume product, volume ratio}


\begin{abstract}
In this paper, we deal with the asymmetric  Orlicz zonotopes by using the method of shadow system. We establish the volume product inequality and volume ratio inequality for asymmetric Orlicz zonotopes, along with their equality cases.
\end{abstract} \maketitle

\section{introduction}
A classical problem in convex geometry is to find the maximizer or minimizer of the volume product among convex bodies. The celebrated Blaschke-Santal\'{o} inequality characterizes ellipsoids are the maximizers of this function on convex bodies. However, finding the minimizer of this function is a trouble in convex geometry. Only in two dimensional case, this problem is solved by Mahler (see, e.g., \cite{mah-ein-min1939,mah-ein-ube1939}). Moreover, it is conjectured by him that simplices are the solution of this function for all dimensional $n$, which is called the Mahler's conjecture. Although it is extremely difficult to attack, but it attracts lots of author's interests, many substantial inroads have been made. One can refer to e.g., \cite{bar-fra-the2013,bou-mil-new1987,cam-gro-on2006,cam-gro-vol2006,fra-mey-zva-an2012,gor-mey-rei-zon1988,hug-sch-rev2011,
kim-rei-loc2011,kup-fro2008,mey-rei-sha2006,rei-zon1986,rog-she-som1958} for more about this conjecture.

One aspect of the researches for the Mahler's conjecture is to make studying the volume product of zonotopes or zonoids, that is the Minkowski sums of origin-symmetric line segments in $\mathbb R^n$, and their limits with respect to the Hausdorff distance (see, e.g., \cite{gor-mey-rei-zon1988,rei-zon1986,sch-wei-zon1983}). Although the restriction to zonotopes and zonoids is a regrettable drawback, but there seems no approach for general convex bodies for this problem. On the other hand, inequalities for zonoids can be applied to stochastic geometry (see \cite{sch-wei-sto2008}).

In the last century, the  volume product inequalities in Euclidean space, $\mathbb R^n$, are widely been generalized with the development of the $L_p$-Minkowski theory. See, for example, \cite{lut-the1993,lut-the1996,lut-yan-zha-lp2005,lut-yan-zha-vol2004,lut-yan-zha-vol2007,lut-yan-zha-lp2000,
wer-ye-new2008,sta-the2002,sta-on2003,
cam-gro-on2002,cam-gro-the2002,cam-gro-on2006,fir-p1962}
for more details about the volume product inequalities with $L_p$-Minkowski theory.
The $L_p$-volume product inequalities for zonotopes, together with its dual volume ratio inequality, were established by Campi and Cronchi \cite{cam-gro-vol2006}. These results extend the results of Reisner \cite{rei-zon1986}. However, all of these results are restricted to the origin-symmetric setting. The asymmetric extension of the $L_p$-volume product inequality and $L_p$-volume ratio inequality, along with the characterization of its extremals are established by Weberndorfer in \cite{web-sha2013}, and the Campi and Cronchi's results as a special case. The seminal work in studying the asymmetric geometric inequalities are very important in convex geometry. For example, in the paper of  Ludwig \cite{lud-min2005}, she's characterization of the asymmetric $L_p$-centroid body and asymmetric $L_p$-projection body, which establishes the classification of the $SL(n)$ invariant Minkowski valuation on convex set. After that, the asymmetric geometric inequalities involving the volume and other geometric invariant are emerged. For instance, the asymmetric $L_p$-centroid body operator turned out to be an extension of $L_p$ version of the Blaschke-Santal\'{o} inequality for all convex bodies, whereas established by Lutwak and Zhang \cite{lut-zha-bla1997} for origin-symmetric setting. One can refer to  \cite{hab-sch-gen2009,hab-sch-asy2009,hab-sch-xia-an2012,sch-web-vol2012} for more details.

Beginning with the articles \cite{hab-lut-yan-zha-the2010,lut-yan-zha-orl-cen2010,lut-yan-zha-orl-pro2010} of Haberl, Lutwak, Yang and Zhang, a more wide extension of the $L_p$-Brunn-Minkowski theory emerged, called the Orlicz Brunn-Minkowski theory. In these papers, the Orlicz Busemann projection inequality and Orlicz Busemann centroid inequality were established. Recently, in a paper of Gardner, Hug and Weil \cite{gar-hug-wei-the2014}, a systematic studies are made on the Orlicz Minkowski addition, the Orlicz Brunn-Minkowski inequality and Orlicz Minkowski inequality are obtained. See, e.g., \cite {bor-str2013,bor-lut-yan-zha-the2012,bor-lut-yan-zha-the2013,che-zho-yan-on2011,che-yan-zho-the2014,gar-hug-wei-the2014,
hab-lut-yan-zha-the2010,hua-he-on2012,zhu-zho-xu-dua2014,zou-xio-orl2014} about the Orlicz Brunn-Minkowski theory.

In view of the importance of the volume product inequality in convex geometry, we tempted to consider the naturally posed problem in the wide interest of the Orlicz Brunn-Minkowski theory.  What is like the volume product inequality or volume ratio inequality for asymmetric Orlicz zonotopes? In this context, the main goal of this paper is to establish the volume product inequality and  volume ratio inequality for asymmetric Orlicz zonotopes.

Throughout this paper, let $ \mathcal C$ be the class of convex, strictly increasing functions $\varphi:[0,\infty)\rightarrow [0,\infty)$ satisfying $\varphi(0)=0$ and $\varphi(1)=1$.

Suppose that $\Lambda$ is a finite set of vectors from $\mathbb R^n\setminus \{o\}$, the asymmetric Orlicz zonotope $Z^+_\varphi\Lambda$ is the unique compact convex set with support function
\begin{align*}
  h_{Z^+_\varphi\Lambda}(u)=\inf\left\{\lambda>0:\sum_{w\in\Lambda}\varphi\left(\frac{\langle w, u\rangle_+}{\lambda}\right)\leq 1\right\}.
\end{align*}
Where $u\in \mathbb R^n$ and $\langle w,u\rangle_+=\max\{0, \langle w,u\rangle\}$ denotes the positive part of the Euclidean scalar product.

Specially, if take $\varphi(t)=t^p$, $p\geq 1$, then $Z^+_\varphi\Lambda$ is precisely the $L_p$-asymmetric zonotope $Z^+_p\Lambda$ defined in \cite{web-sha2013}.

In this paper, our main results are the volume product inequality and volume ratio inequality for the asymmetric Orlicz zonotopes.

Let $\Lambda_\bot=\{e_1,\cdots,e_n\}$ denote the canonical basis of $\mathbb R^n$, $Z^{+,*}_\varphi\Lambda$ denotes the polar body of $Z^{+}_\varphi\Lambda$ with respect to the Santal\'{o} point. For the asymmetric Orlicz zonotopes, we establish the following volume product inequality.
\begin{thm}
  Suppose $\varphi\in \mathcal C$ and $\Lambda$ is a finite and spanning multiset. Then
  \begin{align}
    V(Z^{+,*}_\varphi\Lambda)V(Z^+_1\Lambda)\geq V(Z^{+,*}_\varphi\Lambda_\bot)V(Z^+_1\Lambda_\bot).
  \end{align}
Equality holds with $\varphi\neq Id$ if and only if $\Lambda$ is a $GL(n)$ image of the canonical basis $\Lambda_\bot$. If $\varphi=Id$, the identity function, the equality holds if and only if $Z^+_1\Lambda$ is a parallelepiped.
\end{thm}

We follow the notations of paper \cite{web-sha2013}. A set $\Lambda$ of vectors from $\mathbb R^n$ is called {\it obtuse} if every pair of distinct vectors $u,\,\, v$ from $\Lambda$ satisfies
\begin{align*}
  \langle u,v\rangle_+=0.
\end{align*}

Another result regards to the volume ratio for asymmetric Orlicz zonotopes associate with the obtuse sets $\Lambda$ says that it attains its maximum if $\Lambda$ is a canonical basis of $\mathbb R^n$.
\begin{thm}
 Suppose $\varphi\in\mathcal C$ and $\Lambda$ is a finite and spanning set. Then
  \begin{align*}
    \frac{V(Z^+_\varphi\Lambda)}{V(Z^+_1\Lambda)}\leq\frac{V(Z^+_\varphi \Lambda_\bot)}{V(Z^+_1\Lambda_\bot)}.
  \end{align*}
With equality if and only if $\Lambda$ is a $GL(n)$ image of an obtuse set.
\end{thm}
The paper is organized as follows. In section 2, we introduce the asymmetric Orlicz zonotopes  and show some of their properties. The shadow system and some results of them are given in Section 3. Section 4 deals with the equality case of the volume product inequality and volume ratio inequality for Orlicz zonotopes. The final proofs of the main theorems  are presented in section 5.

\section{preliminaries}
For quick reference we recall some basic definition and notations in convex geometry that is required for our results. Good references see Gardner \cite{gar-geo2006}, Gruber \cite{gru-con2007}, Schneider \cite{sch-con1993}.

Let $\mathcal K^n$ denote the set of convex bodies (compact, convex subsets with nonempty interiors) in Euclidean $n$-space, $\mathbb R^n$. If $K$ is a convex body, denote by $V(K)$ its $n$-dimensional volume, and by $h_K(\cdot):S^{n-1}\rightarrow \mathbb R$
 the support function of $K$; i.e., for
$u\in \mathbb S^{n-1}$,
$$h_K(u)=max\big\{\langle u, x\rangle :x\in K\big\},$$
 where $\langle u, x\rangle$ denotes the standard inner product in $\mathbb R^n$. It is shown that the sublinear support function characterizes a convex body and, conversely, every sublinear function on $\mathbb R^n$ is the support function of a nonempty compact convex set.

 Two convex body $K,\,\,L$ satisfy $K\subseteq L$ if and only if $h_K(\cdot)\leq h_L(\cdot)$. By the definition of the support function, it follows immediately that the support function of the image $\phi K:=\{\phi y:y\in K\}$ is given by
\begin{align*}
h_{\phi K}(x)=h_K(\phi ^Tx)
\end{align*}
for $\phi\in GL(n)$. Here $\phi^T$ denotes the transpose of $\phi$.

Let $K$ be a convex body, for every interior point $s$ of $K$,
\begin{align*}
  K^s=\{y\in \mathbb R^n:\langle y, x-s\rangle\leq 1 \,\,for \,\,\,all\,\, x\in K\}
\end{align*}
defines a convex body that is called the polar body of $K$ with respect to $s$. A well-known result of Santal\'{o}  states that, in every convex body $K\in \mathbb R^n$, there exists a unique point $s(K)\in K$, the Santal\'o point, such that
\begin{align*}
  V(K^{s(K)})=\min_{s\in K}V(K^s).
\end{align*}
To shorten the notation, we shall denote $K^{s(K)}$ by $K^*$. It is well known that the polarization with respect to the Santal\'o point is translation invariant and $GL(n)$ contravariant, that is,
\begin{align*}
  (K+y)^*=K^*\,\,\,\,\,\,\,\,\,and \,\,\,\,\,\,\,\,\,\,(\phi K)^*=\phi^{-T}K^*,
\end{align*}
for $y\in \mathbb R^n$ and $\phi\in GL(n)$.

Suppose $\Lambda$ is a set in $\mathbb R^n$, it is called multiset if its members are allowed to appear more than once. More precisely, a multiset $\Lambda$ is identified with its multiplicity function $1_{\Lambda}:\mathbb R^n\rightarrow N\cup\{0\}$,  that generalizes the characteristic function of sets. We say that a vector is an element of a multiset if the corresponding multiplicity function evaluated at the vector is greater than zero, and call a multiset finite if it contains only a finite number of vectors. If these vectors span $\mathbb R^n$, then we say that the multiset is spanning.

The operation between multisets can be defined using the multiset function. For instance, the union $\Lambda_1\uplus\Lambda_2$ of $\Lambda_1$ and $\Lambda_2$ is defined as
\begin{align*}
  1_{\Lambda_1\uplus \Lambda_2}(x)=1_{\Lambda_1}(x)+1_{\Lambda_2}(x),
\end{align*}
and $\Lambda_1-\Lambda_2$ is defined as
\begin{align*}
  1_{\Lambda_1- \Lambda_2}(x)=\max\{0, 1_{\Lambda_1}(x)-1_{\Lambda_2}(x)\}.
\end{align*}
We write multisets in usual set notation, that is, $\Lambda=\{v_1,\cdots,v_m\}$.

The asymmetric $L_p$-zonotopes associated with finite and spanning multisets $\Lambda=\{v_1,\cdots,v_m\}$ are defined by Weberdorfer \cite{web-sha2013}. Here we extend the notations to asymmetric Orlicz zonotopes.

Let $ \mathcal C$ be the class of convex, strictly increasing functions $\varphi:[0,\infty)\rightarrow [0,\infty)$ satisfying $\varphi(0)=0$ and $\varphi(1)=1$.  Here the normalization is a matter of convenience and other choices are possible. It is not hard to conclude that $\varphi\in \mathcal C$ is continuous on $[0,\infty)$.

Asymmetric Orlicz zonotopes associated with finite and spanning multisets $\Lambda=\{v_1,\cdots,v_m\}$ are defined by
\begin{align}\label{asy-orl-zon}
  h_{Z^+_\varphi\Lambda}(u)=\inf\Bigg\{\lambda>0:\sum_{i=1}^m\varphi\bigg(\frac{\langle v_i,u\rangle_+}{\lambda}\bigg)\leq 1\Bigg\},
\end{align}
for all $u\in S^{n-1}$. Moreover, if $\langle v_i,u\rangle_+=0$ for all $i=1,\cdots,m$, we define $h_{Z^+_\varphi\Lambda}(u)=0$.

In fact, by the convexity of $\varphi$ and the sub-additive of $\langle v, \cdot\rangle_+$, we have
 \begin{align*}
   \varphi\bigg(\frac{\langle v,u_1+u_2\rangle_+}{\lambda_1+\lambda_2}\bigg)\leq\frac{\lambda_1}{\lambda_1+\lambda_2}\varphi\bigg(\frac{\langle v,u_1\rangle_+}{\lambda_1}\bigg)+\frac{\lambda_2}{\lambda_1+\lambda_2}\varphi\bigg(\frac{\langle v,u_2\rangle_+}{\lambda_2}\bigg),
 \end{align*}
 it follows that the support function defined in (\ref{asy-orl-zon}) is sublinear, which grants the existence of convex body $Z^+_\varphi\Lambda$.

Specially, if $p\geq 1$ and take $\varphi(t)=t^p$, then it turns out that $Z^+_\varphi\Lambda=Z^+_p\Lambda$.

Note that $\varphi\in\mathcal C$ is strictly convex and increasing on $[0,\infty)$, it follows that the function
\begin{align*}
  \lambda\mapsto \sum_{i=1}^m\varphi\bigg(\frac{\langle v_i,u\rangle_+}{\lambda}\bigg)
\end{align*}
is strictly decreasing on $[0,\infty)$. The next lemma easily follows.
\begin{lem}\label{orl-nor-lem}
Suppose $\varphi\in \mathcal C$ and $\Lambda=\{v_1,\cdots,v_m\}$ spans $\mathbb R^n$.  For $u_0\in S^{n-1}$, then

(\romannumeral 1),\,\,$\sum_{i=1}^m\varphi\Big(\frac{\langle v_i,u_0\rangle_+}{\lambda_0}\Big)=1$ if and only if $\lambda_0 =h_{Z^+_\varphi\Lambda}(u_0)$;

(\romannumeral 2),\,\,$\sum_{i=1}^m\varphi\Big(\frac{\langle v_i,u_0\rangle_+}{\lambda_0}\Big)>1$ if and only if $\lambda_0 <h_{Z^+_\varphi\Lambda}(u_0)$;

(\romannumeral 3),\,\,$\sum_{i=1}^m\varphi\Big(\frac{\langle v_i,u_0\rangle_+}{\lambda_0}\Big)<1$ if and only if $\lambda_0 >h_{Z^+_\varphi\Lambda}(u_0)$.
\end{lem}

A simple observe of definition (\ref{asy-orl-zon}) is that the operator $Z^+_\varphi$ on finite and spanning multisets is $GL(n)$ equivariant, that is,
$Z^+_\varphi \phi\Lambda=\phi Z^+_\varphi\Lambda$ holds for all $\phi\in GL(n)$. In fact,
\begin{align*}
 h_{Z^+_\varphi\phi\Lambda}(u)&=\inf\Bigg\{\lambda>0:\sum_{i=1}^m\varphi\bigg(\frac{\langle \phi v_i, u\rangle_+}{\lambda}\bigg)\leq 1\Bigg\}\\
 &=\inf\Bigg\{\lambda>0:\sum_{i=1}^m\varphi\bigg(\frac{\langle v_i,\phi^T u\rangle_+}{\lambda}\bigg)\leq 1\Bigg\}\\
 &=h_{Z^+_\varphi\Lambda}(\phi^Tu)=h_{\phi Z^+_\varphi\Lambda}(u),
\end{align*}
holds for all $u\in S^{n-1}$. Moreover, the asymmetric Orlicz zonotopes defined by (\ref{asy-orl-zon}) are closely related to the origin symmetric Orlicz zonotopes $Z_\varphi \Lambda$ defined in \cite{wan-len-hua-vol2012}.
Specially,
\begin{align*}
\varphi\bigg(\frac{\rvert\langle v_i, u\rangle\lvert}{\lambda}\bigg)&=\varphi\bigg(\frac{\langle v_i, u\rangle_++\langle -v_i, u\rangle_+}{\lambda}\bigg)\\
&=\varphi\bigg(\frac{\langle v_i, u\rangle_+}{\lambda}\bigg)+\varphi\bigg(\frac{\langle -v_i, u\rangle_+}{\lambda}\bigg).
\end{align*}
 Which implies that $Z_\varphi\Lambda=Z^+_\varphi(\Lambda\uplus-\Lambda)$.

In the following, the volume product and volume ratio for Orlicz zonotopes associate with the multisets $\Lambda$ always refer to $V(Z^{+,*}_\varphi\Lambda)V({Z^+_1\Lambda})$ and $\frac{V(Z^+_\varphi \Lambda)}{V(Z^+_1\Lambda)}$, respectively.

\section{shadow system of multiset}

The notation of shadow system, introduced by Rogers and Shephard (see \cite{rog-she-som1958,she-sha1964}), play an important role in proving geometric  inequalities in convex geometry. For example, this method was used by Campi, Gronchi, Meyer and Reisner (see, e.g., \cite{cam-col-gro-a1999,cam-gro-the2002,cam-gro-on2002,cam-gro-on2006,cam-gro-vol2006,mey-rei-sha2006}).
 A shadow system $X_t$ of points from $\mathbb R^n$ is a family of sets which can be defined as follows:
\begin{align*}
  X_t=\{x_i+t\beta_i v: x_i \in\mathbb R^n\}
\end{align*}
 where $t\in[t_1,t_2]$, $\beta_i\in \mathbb R$. Here $t$ can be seen as a time-like parameter and $\beta_i$ as the speed of the point $x_i$ along the direction $v$.

 Let $\Lambda=\{v_1,\cdots, v_m\}$ be a finite multiset such that $\Lambda\setminus v_1$ is spanning. Following the ideas of Campi and Gronchi \cite{cam-gro-vol2006}, define $\Lambda^a_t=\{w_1(t),\cdots,w_m(t)\}$, where
 \begin{align}\label{ort-sha-sys}
  w_i(t)=\left\{
      \begin{array}{ll}
        (1+t a) v_1, & \hbox{i=1;} \\
        v_i-t\frac{\langle v_1,v_i\rangle}{\lVert v_1\rVert^2}v_1, & \hbox{otherwise.}
      \end{array}
    \right.
 \end{align}
Where $t$ varies in $[-a^{-1},1]$, and
\begin{align}\label{sha-sys-con}
a=\frac{\sum\limits_{2\leq i_1<\cdots<i_n\leq m}\big|[v_{i_1},\cdots ,v_{i_n}]\big|}{\sum\limits_{2\leq i_2<\cdots<i_n\leq m}\big|[v_{1},v_{i_2}\cdots ,v_{n}]\big|},
 \end{align}
here $[v_{i_1},\cdots ,v_{i_n}]$ denotes the determinant of the matrix whose rows are $v_{i_1},\cdots ,v_{i_n}$. From the definition we have $\Lambda^a_0=\Lambda$, and $w_1(1)$ is orthogonal to the remaining vectors in $\Lambda_1^a$, while $w_1(-a^{-1})=o$. Moreover, by the construction (\ref{ort-sha-sys}), $\Lambda^a_t$, $t\in[-a^{-1},1]$, is a shadow system of multisets along the direction $v=\frac{v_1}{\lVert v_1\rVert}\in S^{n-1}$.

The important result of Campi and Gronchi is that for $t\in [-a^{-1},1]$, the asymmetric $L_1$-zonotopes associate $\Lambda^a_t$ preserve the volume, and then be extended to asymmetric $L_p$-zonotopes by Weberndorfer. From now on, we use $\Lambda_t$ to denote $\Lambda^a_t$, the orthogonalization of $\Lambda$ with respect to $v_1$ if $a$ is determined by (\ref{sha-sys-con}).

In the following, we will show that the asymmetric Orlicz zonotopes associate with the shadow system of a multiset along direction $v$ is independent of $t$.
\begin{lem}\label{pro-ind-lem}
Suppose that $\Lambda_t$, $t\in[-a^{-1}, 1]$ is a shadow system of multisets along the
direction $v\in S^{n-1}$. Then the orthogonal projection of $Z_{\varphi}^{+}\Lambda_t$ onto $v^\bot$ is independent of $t$.
\end{lem}
\begin{proof}
  By the definition of $\Lambda_t$, for $x\in{v^\bot}$ we have
  \begin{align*}
    h_{Z^+_\varphi \Lambda_t}(x)&=\inf\left\{\lambda>0: \sum^m_{i=1}\varphi\left(\frac{\langle v_i+t\beta_i v,x\rangle_+ }{\lambda}\right)\leq 1\right\}\\
    &=\inf\left\{\lambda>0: \sum^m_{i=1}\varphi\left(\frac{\langle v_i,x\rangle_+ }{\lambda}\right)\leq 1\right\}\\
    &=h_{Z^+_\varphi \Lambda}(x).
  \end{align*}
Which shows the result.
\end{proof}

Before characterizing the shadow system of convex bodies along with direction $v$ distinguish with others, we introduce the uppergraph function $\overline g_v(K, \cdot)$ and the lowergraph function $\underline g_v(K, \cdot)$ of a convex body $K$.
\begin{align*}\begin{split}
\overline g_v(K, x):=\sup\{\lambda\in \mathbb R: x+\lambda v\in K\};\\
\underline g_v(K, x):=\inf\{\lambda\in \mathbb R: x+\lambda v\in K\}.
\end{split}\end{align*}
An alternative representation of the above formulas are obtain by Weberndorfer \cite{web-sha2013}.
Let $w\in v^\bot$, then
\begin{align}\label{up-low-gra}
\begin{split}
 \overline g_v(K,x)=\inf_{w\in v^\bot}\{h_K(v+w)-\langle x, w\rangle\};\\
\underline g_v(K, x)=-\inf_{w\in v^\bot}\{h_K(-v-w)+\langle x, w\rangle\}.
\end{split}
\end{align}
for all $x\in v^\bot$.

Now we present the characterization of a shadow system obtain by Campi and Gronchi.
\begin{prop}\cite{cam-gro-on2002}\label{cam-gro-thm}
  Let $K_t$, $t\in[-a^{-1},1]$, be one parameter family of convex bodies such that $K_t|_{v^\bot}$ is independent of $t$. Then $K_t$, $t\in[-a^{-1},1]$, be a shadow system of convex bodies along the direction $v$ if and only if for every $x\in K_0|_{v^\bot}$, the functions $t\rightarrow \overline g_v(K_t,x)$ and $t\rightarrow -\underline{g}_v(K_t,x)$ are convex and
  \begin{align}\label{upp-low-ine}
    \underline g_v(K_{\lambda s+\mu t}, x)\leq \lambda\overline g_v(K_s,x)+\mu \underline g_v(K_t,x)\leq\overline g_v(K_{\lambda s+\mu t}, x)
  \end{align}
  for every $s, t\in[-a^{-1},1]$ and $\lambda,\mu\in (0,1)$ such that $\lambda+\mu=1$.
\end{prop}

A remarkable result about the volume of a shadow system is due to Shephard.
\begin{lem}\cite{she-sha1964} \label{vol-con-lem}
Every mixed volume involving $n$ shadow systems along the same
direction is a convex function of the parameter. In particular, the
volume $V(K_t)$ and all quermassintegrals
$W_i(K_t),\,\,i=1,\,2,\,\cdots, n$, of a shadow system are convex
functions of $t$.
\end{lem}
This result was largely used by Campi and Gronchi \cite{cam-gro-the2002,cam-gro-on2002,cam-gro-on2006,cam-gro-vol2006}, Li \cite{li-len-a2011} and Chen \cite{che-zho-yan-on2011}.  In the following, we show that the support function of asymmetric Orlicz zonotopes associate with a shadow system of multisets $\Lambda_t$, $t\in[-a^{-1},1]$,  is a Lipschitz function of t.

 Let $f=(f_1,\cdots,f_n):\mathbb R^n\rightarrow \mathbb R$ be a real value function, $\overline v=(v_1,\cdots,v_n)$, $\beta=(\beta_1,\cdots,\beta_n)$ and $v(t)=(v_1(t),\cdots, v_n(t))$ are vectors in $\mathbb R^n$, where $v_i(t)=v_i+t\beta_i v$ for $i=1,\cdots,n$,  as defined before. For notational convenience we define
  \begin{align}\label{orl-nor-def}
    \lVert f(t)\rVert_\varphi:=\inf\left\{\lambda>0; \sum^m_{i=1}\varphi\Big(\frac{\lvert f_i(t)\rvert}{\lambda}\Big)\leq 1\right\},
  \end{align}
for real-valued functions $f$ on $\mathbb R^n$, and $[\cdot]_+:=\max\{\cdot, 0\}$. By Lemma \ref{orl-nor-lem}, if $c>0$, we have $\lVert c f\rVert=\lvert c\rvert \lVert f\rVert$. Moreover, if $f\leq g$ for all $t\in \mathbb R^n$, we have
\begin{align} \lVert f\rVert_\varphi\leq \lVert g\rVert_\varphi.
\end{align}
\begin{lem}\label{lip-sup-lem}
Suppose $\varphi\in\mathcal C$ and $\Lambda_t$, $t\in [-a^{-1}, 1]$, is a shadow system of multiset $\Lambda=\{v_1,\cdots,v_m\}$ along the direction $v$ and speed function $\beta$. If $t_1, t_2\in [-a^{-1}, 1]$ and $x\in \mathbb R^n$, then
 \begin{align*}
   \big| h_{Z_{\varphi}^{+}\Lambda_{t_1}}(x)- h_{Z_{\varphi}^{+}\Lambda_{t_2}}(x)\big|\leq \lVert \beta\langle v, x\rangle \rVert_\varphi \lvert t_1-t_2\rvert.
 \end{align*}

\end{lem}
 \begin{proof}
   From definition (\ref{orl-nor-def}) and Lemma \ref{orl-nor-lem}, we have
   \begin{align*}\begin{split}
     \rVert f(t)\lVert_\varphi=\lambda_1 \Leftrightarrow \sum^{m}_{i=1}\varphi\Big(\frac{|f_i(t)|}{\lambda_1}\Big)=1; \\
     \rVert g(t)\lVert_\varphi=\lambda_2 \Leftrightarrow \sum^{m}_{i=1}\varphi\Big(\frac{|g_i(t)|}{\lambda_2}\Big)=1.
\end{split}
\end{align*}
   The convexity of $\varphi$  shows
   \begin{align}\label{con-lam}
     \varphi\left(\frac{|f_i(t)+g_i(t)|}{\lambda_1+\lambda_2}\right)\leq\frac{\lambda_1}{\lambda_1+\lambda_2}
     \varphi\left(\frac{|f_i(t)|}{\lambda_1}\right)+\frac{\lambda_2}{\lambda_1+\lambda_2}
     \varphi\left(\frac{|g_i(t)|}{\lambda_2}\right).
   \end{align}
 Summing both sides of (\ref{con-lam}) with respect to $i=1,\cdots, m$, gives
 \begin{align*}
\sum^m_{i=1}\varphi\left(\frac{|f_i(t)+g_i(t)|}{\lambda_1+\lambda_2}\right) \leq 1.
 \end{align*}
 By Lemma \ref{orl-nor-lem} again we have
 \begin{align}\label{orl-tri-ine}
 \lVert f(t)+g(t)\lVert_\varphi\leq  \lVert f(t)\lVert_\varphi+\lVert g(t)\lVert_\varphi.
 \end{align}
If we take $f=f-g+g$, we have
\begin{align*}
  \lVert f(t)\lVert_\varphi-\lVert g(t)\lVert_\varphi\leq  \lVert f(t)-g(t)\lVert_\varphi.
\end{align*}
Which means
\begin{align}\label{tri-orl}
 \big| \lVert f(t)\lVert_\varphi-\lVert g(t)\lVert_\varphi\big|\leq  \lVert f(t)-g(t)\lVert_\varphi.
\end{align}
Moreover, together with the definition of the support function $h_{Z^+_\varphi \Lambda_t}(x)$ and (\ref{orl-nor-def}), we have
$$h_{Z_{\varphi}^{+}\Lambda_t}(x)=\lVert \langle v(t), x\rangle_+\rVert_\varphi.$$
Then, by (\ref{tri-orl}) we have
\begin{align}\begin{split}
\big|h_{Z_{\varphi}^{+}\Lambda_{t_1}}(x)-h_{Z_{\varphi}^{+}\Lambda_{t_2}}(x)\big|&=
\big|\lVert \langle v(t_1), x\rangle_+\rVert_\varphi-\lVert \langle v(t_2), x\rangle_+\rVert_\varphi\big|\\
&\leq\lVert \langle v(t_1), x\rangle_+-\langle v(t_2), x\rangle_+\rVert_\varphi\\
&\leq\lVert \beta\langle x,v\rangle_+(t_1-t_2)\rVert_\varphi=\lVert \beta\langle x,v\rangle_+\rVert_\varphi|t_1-t_2|.
\end{split}
\end{align}
We complete the proof.
\end{proof}

\begin{thm}\label{sha-orl-thm}
Suppose $\varphi\in \mathcal C$, $\Lambda_t$, $t\in [-a^{-1},1]$, is a shadow system of multisets along the direction $v\in S^{n-1}$. Then $Z^+_\varphi \Lambda_t$, $t\in [-a^{-1},1]$, is a shadow system of convex bodies along the direction $v$.
\end{thm}
\begin{proof}
  Let $x$ be a point in $Z^+_\varphi \Lambda_0|_{v^\bot}$, and $\nu, \mu\in (0,1)$ satisfy $\nu+\mu=1$. By Lemma \ref{pro-ind-lem}, it is remains to show that the hypotheses of Proposition \ref{cam-gro-thm} on properties of the graph functions are satisfied.

  By assumption, the shadow system $\Lambda_t$ is equal to, say, $\{v_1(t),\cdots,v_m(t)\}$ where $v_i(t)=v_i+t\beta_iv$.  With these definitions, the support function of $Z_\varphi^+\Lambda_t$ can be written as
\begin{align*}
h_{Z_\varphi^+\Lambda_t}(u)=\inf\bigg\{\lambda>0: \sum^{m}_{i=1}\varphi\Big(\frac{\langle v_i(t), u\rangle_+}{\lambda}\Big)\leq 1\bigg\}=\lVert \langle v(t), u\rangle_+\rVert_\varphi.
\end{align*}
To  establish the convexity of the uppergraph and lowergraph function as  functions of $t$,  we firstly prove the uppergraph function is a convex function of $t$. Notice that $\Lambda_t$ is also a shadow system in direction $-v$. Then the vector $v$ can be replaced by $-v$ and by application of the identity $\overline g_{-v}(\cdot,x)=-\underline g_v(\cdot, x)$, we can obtain that the lowergraph function is a convex function of $t$.

By the definition of $\overline g_v(Z_\varphi^+\Lambda_{t}, x)$, we have
\begin{align*}
\nonumber\overline g_v(Z_\varphi^+\Lambda_{\nu s+\mu t}, x)&=\inf_{w_1,w_2\in v^\bot}\left\{h_{Z_\varphi^+\Lambda_{\nu s+\mu t}}(v+\nu w_1+\mu w_2)-\langle x, \nu w_1+\mu w_2\rangle\right\}\\
&=\inf_{w_1,w_2\in v^\bot}\Big\{\lVert \langle v(
\nu s+\mu t), v+\nu w_1+\mu w_2\rangle_+\rVert_\varphi-\langle x, \nu w_1+\mu w_2\rangle\Big\}.
\end{align*}
By the inequality $\max\{u+v,0\}\leq\max\{u,0\}+\max\{v,0\}$ and definition (\ref{orl-nor-def}), we have
 \begin{align*}
  & \lVert \langle v(\nu s+\mu t),v+\nu w_1+\mu w_2\rangle_+\lVert_\varphi\\
   &=\inf\left\{\lambda>0: \sum^{m}_{i=1}\varphi\bigg(\frac{\langle v_i+(\nu s +\mu t)\beta_iv, v+\nu w_1+\mu w_2\rangle_+}{\lambda}\bigg)\leq 1\right\}.
   \end{align*}
The convexity of $\varphi$ and (\ref{orl-tri-ine}) imply that
 \begin{align}\label{orl-nor-ine}
\lVert \langle v(\nu s+\mu t),v+\nu w_1+\mu w_2\rangle_+\lVert_\varphi\leq\nu\lVert\langle v(s),v+w_1\rangle_+\lVert_\varphi+\mu\lVert\langle v(t),v+w_2\rangle_+\lVert_\varphi.
 \end{align}
 Thus, together with (\ref{orl-nor-ine}) and the expression of $\overline g_v(K,x)$ we have
 \begin{align}\begin{split}
\inf_{w_1,w_2\in v^\bot}&\Big\{\lVert \langle v(
\nu s+\mu t), v+\nu w_1+\mu w_2\rangle_+\rVert_\varphi-\langle x, \nu w_1+\mu w_2\rangle\Big\}\\
&\leq \inf_{w_1\in v^\bot}\Big\{\nu\lVert \langle v(s), v+ w_1\rangle_+\rVert_\varphi-\nu\langle x, w_1\rangle\Big\}\\
&+\inf_{w_2\in v^\bot}\Big\{\mu\lVert \langle v(t), v+ w_2\rangle_+\rVert_\varphi-\mu\langle x, w_2\rangle\Big\}.
\end{split} \end{align}
Which means
\begin{align}\label{upp-orl-fun}
  \overline g_v(Z_\varphi^+\Lambda_{\nu s+\mu t}, x)\leq \nu  \overline g_v(Z_\varphi^+\Lambda_{ s}, x)+\mu  \overline g_v(Z_\varphi^+\Lambda_{t}, x).
\end{align}
Hence $t\rightarrow  \overline g_v(Z_\varphi^+\Lambda_{ t}, x)$ is convex.

Next we verify the inequality (\ref{upp-low-ine}) of Proposition \ref{cam-gro-thm}. First we show
\begin{align}\label{rig-ineq}
  \nu \overline g_v(Z^+_\varphi \Lambda_s,x)+ \mu \underline g_v(Z^+_\varphi \Lambda_t,x)\leq \overline g_v(Z^+_\varphi \Lambda_{\nu s+\mu t},x).
\end{align}
To see this, let $w\in v^\bot$,
\begin{align}\label{low-ine}
\nu\overline g_v(Z^+_\varphi \Lambda_s,x)=\inf_{w\in v^\bot}\Big\{\lVert \nu\langle v(s),v+w\rangle_+\rVert_\varphi-\nu\langle x, w\rangle\Big\}.
\end{align}
Let $w=\nu^{-1}(w_1-\mu w_2)$, $w_1, w_2\in v^\bot$, in (\ref{low-ine}), we have
\begin{align}
\nu\overline g_v(Z^+_\varphi \Lambda_s,x)=\inf_{w_1,w_2\in v^\bot}\Big\{\lVert \langle v(s),\nu v+w_1-\mu w_2\rangle_+\rVert_\varphi-\langle x, w_1-\mu w_2\rangle\Big\},
\end{align}
where
\begin{align*}
\langle v(s),\nu v&+w_1-\mu w_2\rangle_+= \langle v_i+s\beta_iv,(1-\mu) v+w_1-\mu w_2\rangle_+\\
&=\mu\langle v_i+ t\beta_iv, -v-w_2\rangle_++\langle v_i+(\nu s+\mu t)\beta_iv, v+w_1\rangle_+.
\end{align*}
By (\ref{orl-tri-ine}) we have,
\begin{align*}
\nu\overline g_v&(Z^+_\varphi \Lambda_s,x)=\inf_{w_1,w_2\in v^\bot}\Big\{\lVert \langle v(s),\nu v+w_1-\mu w_2\rangle_+\rVert_\varphi-\langle x, w_1-\mu w_2\rangle\Big\}\\
&=\inf_{w_1,w_2\in v^\bot}\Big\{\lVert \mu\langle v_i+ t\beta_iv, -v-w_2\rangle_++\langle v_i+(\nu s+\mu t)\beta_iv, v+w_1\rangle_+\rVert_\varphi-\langle x, w_1-\mu w_2\rangle\Big\}\\
&\leq\mu\inf_{w\in v^\bot}\Big\{\lVert\langle v(t), -v-w_2\rangle_+\rVert_\varphi-\langle x, - w_2\rangle\Big\}+\inf_{w\in v^\bot}\Big\{\lVert\langle v(\nu s+\mu t), v+w_1\rangle_+\rVert_\varphi-\langle x, w_1\rangle\Big\}\\
&=-\mu\underline g_v(Z^+_\varphi\Lambda_t,x)+\overline g_v(Z^+_\varphi \Lambda_{\nu s+\mu t},x).
\end{align*}
Which implies the inequality (\ref{rig-ineq}).

The left hand of inequality (\ref{upp-low-ine}) can be derived from inequality (\ref{rig-ineq}) by replacing $v$ by $-v$, and using the following facts
\begin{align*}
  \underline g_{-v}(\cdot,x)=-\overline g_v(\cdot, x) \,\,\,\,\,\,\, and \,\,\,\,\,\,\,\,\overline g_{-v}(\cdot,x)=-\underline g_v(\cdot, x).
\end{align*}
Now we complete the proof.
\end{proof}
Now Theorem \ref{sha-orl-thm} together with Lemma \ref{pro-ind-lem} imply the following Theorem.
\begin{thm}\label{orl-vol-rat-thm}
  Suppose $\varphi\in\mathcal C$ and $\Lambda$ is a finite and spanning multiset. If $\Lambda_t$, $t\in [-a^{-1},1]$, is an orthogonalization of $\Lambda$ defined by (\ref{ort-sha-sys}), then:

 (\romannumeral 1) The volume $V(Z^{+,*}_\varphi \Lambda_t)^{-1}$ is a convex function of $t$. In particular, the inverse volume product for asymmetric Orlicz zonotopes associated with $\Lambda_t$ is a convex function of $t$.

  (\romannumeral 2)  The volume $V(Z^{+}_\varphi \Lambda_t)^{-1}$ is a convex function of $t$. In particular, the volume ratio for asymmetric Orlicz zonotopes associate with $\Lambda_t$ is a convex function of $t$.
 \end{thm}

Theorem \ref{orl-vol-rat-thm} shows that the inverse volume product and the volume ratio for asymmetric Orlicz zonotopes  are nondecreasing if $\Lambda$ is replaced by either $\Lambda_{-a^{-1}}$ or $\Lambda_1$, because convex functions attain global maxima at the boundary of compact intervals.

\section{the equality condition}

The following lemma is crucial for our proof of main results.
\begin{lem}\label{spa-obt-lem}
  Suppose $\varphi\in \mathcal C $, and $\Lambda$ is a finite and spanning multiset. Replace all vectors in $\Lambda$ that point in the same direction by their sum, and denote this new multiset by $\overline \Lambda$. Then the following inequalities
  \begin{align}\begin{split}
    \label{orl-vol-pro-eq}
    \frac{V(Z^+_\varphi\Lambda)}{V(Z^+_1 \Lambda)}\leq \frac{V(Z^+_\varphi\overline\Lambda)}{V(Z^+_1 \overline\Lambda)},\\
  V(Z^{+,*}_\varphi\Lambda)V(Z^+_1 \Lambda)\geq(V^{+,*}_\varphi\overline \Lambda)(V(Z^+_1 \overline\Lambda),
  \end{split}
\end{align}
hold. With equalities, when $\varphi\neq Id$, if and only if $\Lambda=\overline \Lambda$.
\end{lem}
\begin{proof}
Let multiset $\Lambda=\{v_1,\cdots,v_m\}$, and $\overline \Lambda=\{w_1,\cdots,w_k\}$. By the construction of $\overline \Lambda$ we have
\begin{align*}
  w_j=\sum_{i\in I_j}v_i,
\end{align*}
where $I_j$, $j=1,\cdots k$, is a partition of $\{1,\cdots m\}$, and the vectors in every $\{v_i:i\in I_j\}$ point in the same direction. If $\varphi=Id$, that means $\varphi(t)=t$, here we write $Z^+_{Id}\Lambda=Z^+_1\Lambda$. It is easy to know that
\begin{align}
\label{eur-zon-pro}Z_1^+\Lambda=Z_1^+\overline\Lambda.\end{align}

Now assume that $\varphi\neq Id$. Let $u\in S^{n-1}$, and set
\begin{align*}
  h_{Z^+_\varphi\Lambda}(u)=\lambda \,\,\,\,\,\,and \,\,\,\,\,\,\, h_{Z^+_\varphi\overline\Lambda}(u)=\overline\lambda.
\end{align*}
By the definition of $Z^+_\varphi\overline\Lambda$, and note that all $v_i$ point in the same direction for $i\in I_j$,  then we have
 $$\langle \sum_{i\in I_j}v_i,u\rangle_+=\sum_{i\in I_j}\langle v_i,u\rangle_+.$$ It follows that
\begin{align*}
 \sum^{k}_{j=1}\varphi\left(\frac{\langle\sum_{i\in I_j} v_i,u\rangle_+}{\overline\lambda}\right)=\sum^{k}_{j=1}\varphi\left(\frac{\sum_{i\in I_j}\langle v_i,u\rangle_+}{\overline\lambda}\right)=1.
\end{align*}
Since the fact that $\varphi$ is convex and increasing, we have that if $x_1,\cdots, x_l\in[0,\infty)$, then
\begin{align}\label{con-fun}
\varphi(x_1+\cdots+x_l)\geq \varphi(x_1)+\cdots+\varphi(x_l),
\end{align}
with equality if and only if $\varphi$ is a linear function.
So we obtain
\begin{align}\label{orl-var-equ}
1=\sum_{j=1}^k\varphi\left(\frac{\sum_{i\in I_j}\langle v_i,u\rangle_+}{\overline\lambda}\right)\geq\sum_{i=1}^m\varphi\left(\frac{\langle v_i,u\rangle_+}{\overline\lambda}\right).
\end{align}
Since $ h_{Z^+_\varphi\Lambda}(u)=\lambda$, by Lemma \ref{orl-nor-lem}, we obtain
$\lambda\leq\overline \lambda$. Since $\varphi\neq Id$, with equality holds only if all sum over $i\in I_j$ contain at most one positive summand, that means $\Lambda=\overline\Lambda$. In fact, if $\Lambda\neq\overline \Lambda$, say $v_1$ and $v_2$ point in the same direction, by the convexity of $\varphi$, then $h_{Z^+_\varphi \overline\Lambda}(v_1)> h_{Z^+_\varphi \Lambda}(v_1)$. Hence we obtain
$Z^+_\varphi \Lambda\subset Z^+_\varphi \overline\Lambda.$

On the other hand, if $\Lambda\neq \overline\Lambda$, with equality in (\ref{orl-var-equ}) if and only if
\begin{align*}
\sum_{j=1}^k\varphi\left(x_j\right)=\varphi\left(\sum_{i=1}^k x_i\right),
\end{align*}
holds for arbitrary $k$ and $x_i\in \mathbb R$. Combine  with the convexity and the normalization of $\varphi$, and solve this functional equation we know that $\varphi(t)=t$.  Then $$Z^{+}_\varphi \Lambda=Z^+_1\Lambda=Z^+_1\overline\Lambda.$$
The first inequality  of (\ref{orl-vol-pro-eq}) now follows immediately. To the second inequality of (\ref{orl-vol-pro-eq}), if $\varphi\neq Id$, note that
\begin{align*}
  Z^{+,*}_\varphi\Lambda=(Z^{+}_\varphi\Lambda-s(Z^{+}_\varphi\Lambda))^o\supseteq(Z^{+}_\varphi\overline\Lambda-s(Z^{+}_\varphi\Lambda))^o,
\end{align*}
with equality if and only if $\Lambda=\overline\Lambda$. Thus
\begin{align*}
  V(Z^{+,*}_\varphi\Lambda)\geq V((Z^{+}_\varphi\overline\Lambda-s(Z^{+}_\varphi\Lambda))^o)\geq V(Z^{+,*}_\varphi\overline\Lambda).
\end{align*}
Together with (\ref{eur-zon-pro}) show the second inequality of (\ref{orl-vol-pro-eq}).
\end{proof}

In the following, we observe that a set that can be written as a disjoint union $\Lambda_\bot\cup \{v_1,\cdots,v_l\}$ is obtuse if and only if there are disjoint nonempty subset $I_1,\cdots, I_l$ of $\{1,\cdots,n\}$ and positive numbers $\mu_i$ such that, for every $j\in\{1,\cdots,l\}$,
\begin{align*}
  v_j=\sum_{i\in I_j}-\mu_ie_i.
\end{align*}

The following Lemma shows that every spanning obtuse set has a linear image of above type.
\begin{lem}\cite{web-sha2013}\label{spa-mul-lem}
Suppose $\Lambda$ is a spanning obtuse set, then the following three statements holds:

(\romannumeral 1) If $B\subset A$ is a basis, then the vectors in $A\setminus B$ are pairwise orthogonal and have nonpositive components with respect to the basis $B$.

(\romannumeral 2) Every $GL(n)$ image of $\Lambda$ that contains the canonical basis $\Lambda_\bot$ is obtuse.

(\romannumeral 3) Suppose in addition that $\Lambda$ contains the canonical basis. For every $y\in Z^+_\varphi\Lambda$ there is a $\phi\in GL(n)$ such that $\phi y$ has nonnegative coordinates with respect to the canonical basis and $\Lambda_\bot \subset \phi \Lambda$.
\end{lem}

This Lemma is established by Webermdorfer in \cite{web-sha2013}, here we omit the proof this lemma.

One of the immediate implications of the above Lemma is that a spanning obtuse set contains at least $n$ and not more than $2n$ vectors. Now we give the equality condition of our main results.
\begin{lem}\label{orl-vol-lem}
  Suppose $\varphi\in\mathcal C$ and $\Lambda$ is a spanning obtuse set. Then
  \begin{align*}
    \frac{V(Z^+_\varphi\Lambda)}{V(Z^+_1\Lambda)}=\frac{V(Z^+_\varphi\Lambda_\bot)}{V(Z^+_1\Lambda_\bot)}.
  \end{align*}
\end{lem}
\begin{proof}
 Let  $\Lambda$ be a spanning obtuse set. The $GL(n)$ invariance of the volume ratio for asymmetric Orlicz zonotopes and Lemma \ref{spa-mul-lem}, we may assume that $\Lambda=\{w_1,\cdots,w_{m}\}$, where $n\leq m\leq 2n$, contains the canonical basis $\Lambda_\bot=\{e_1,\cdots,e_n\}$.
 In the following, if we can establish the dissection formula
  \begin{align}\label{dis-orl-asy}
Z^+_\varphi\Lambda=\bigcup_{1\leq i_1<\cdots <i_n\leq m}Z^+_\varphi\big\{w_{i_1},\cdots,w_{i_n}\}.
  \end{align}
Then, we have
\begin{align} \label{vol-dis}
  V(Z^+_\varphi \Lambda)=\sum_{1\leq i_1<\cdots<i_n\leq m}V\Big(Z^+_\varphi\{v_{i_1},\cdots,v_{i_n}\}\Big).
\end{align}
The $GL(n)$ equivariance of $Z^+_\varphi $  together with (\ref{vol-dis}) for $\varphi(t)=t$, we have
\begin{align*}
  \frac{V(Z^+_\varphi \Lambda)}{V(Z^+_1 \Lambda)}=\frac{\sum\limits_{1\leq i_1<\cdots<i_n\leq m}V\Big(Z^+_1\{v_{i_1},\cdots,v_{i_n}\}\Big)}{\sum\limits_{1\leq i_1<\cdots<i_n\leq m}V\Big(Z^+_1\{v_{i_1},\cdots,v_{i_n}\} \Big)}=\frac{V(Z^+_\varphi \Lambda_\bot)}{V(Z^+_1 \Lambda_\bot)}.
\end{align*}
Here we used the $GL(n)$ equivariance of $Z^+_\varphi$ and the fact $V\big(Z^+_\varphi \{v_{i_1},\cdots v_{i_n}\}\big)=0$, if $\{v_{i_1},\cdots v_{i_n}\}$ is not a $GL(n)$ image of conical basis $\Lambda_\bot$. Hence we have
\begin{align*}
\frac{V(Z^+_\varphi \Lambda)}{V(Z^+_1\Lambda)}=\frac{V(Z^+_\varphi \Lambda_\bot)}{V(Z^+_1 \Lambda_\bot)}.
\end{align*}

In the following, we will show the dissection formula (\ref{dis-orl-asy}) holds.
Let $y\in\bigcup_{1\leq i_1<\cdots <i_n\leq m}Z^+_\varphi\big\{w_{i_1},\cdots,w_{i_n}\}$, it must belong to, we say, $Z^+_\varphi\{w_1,\cdots,w_n\}$. In order to prove $y\in Z^+_\varphi\Lambda$. Let
\begin{align*}
h_{Z^+_\varphi\{w_1,\cdots,w_n\}}(u)=\lambda_0\,\,\,\,\,\,\,and \,\,\,\,\,\,\,h_{Z^+_\varphi\Lambda}(u)=\lambda_1.
\end{align*}
 By the definition of the support function, we have
\begin{align*}
 \sum_{i=1}^{n}\varphi\left(\frac{\langle w_i,u\rangle_+}{\lambda_0}\right)= 1 \,\,\,\,\,\,\,\,\, and\,\,\,\,\,\,\,\,\,\,\,
\sum_{j=1}^{m}\varphi\left(\frac{\langle w_j,u\rangle_+}{\lambda_1}\right)= 1.
\end{align*}
Since $\varphi$ is increasing, which implies
 \begin{align*}
  \sum_{i=1}^{n}\varphi\left(\frac{\langle w_i,u\rangle_+}{\lambda_1}\right)\leq\sum_{j=1}^{m}\varphi\left(\frac{\langle w_j,u\rangle_+}{\lambda_1}\right).
\end{align*}
By Lemma \ref{orl-nor-lem} we have $\lambda_1\geq \lambda_0$. That means
\begin{align}\label{asy-con-equ}
Z^+_\varphi\{w_1,\cdots,w_n\}\subseteq Z^+_\varphi\Lambda.
\end{align}
We prove $Z^+_\varphi\Lambda$ contains the right hand side of (\ref{dis-orl-asy}).
Now it remains to prove that $Z^+_\varphi\Lambda$ is a subset of the right hand side of (\ref{dis-orl-asy}).
Let $y\in  Z^+_\varphi\Lambda$, it is sufficient to show that there is a $\phi\in GL(n)$ such that $y\in Z^+_\varphi \phi^{-1}\Lambda_\bot$ and $\phi^{-1}\Lambda_\bot\subseteq \Lambda$. By Lemma \ref{spa-mul-lem}, there is a $\phi\in GL(n)$ such that $\phi y$ has nonnegative coordinates with respect to the canonical basis and $\Lambda_\bot\subseteq \phi\Lambda$.  Moreover, $\phi \Lambda$ is obtuse, then we can write
\begin{align*}
  \phi \Lambda=\Lambda_\bot\cup\{w_1,\cdots, w_{m-n}\},
\end{align*}
 and there are disjoint subsets $I_1,\cdots, I_{m-n}$, and positive number $\mu_i$ such that, for $1\leq j\leq m-n$,
\begin{align}\label{neg-com}
  w_j=\sum_{i\in I_j}-\mu'_ie_{i}.
\end{align}
 Let $h_{Z^+_\varphi \phi\Lambda}=\lambda_0$.
Then we have
\begin{align*}
 \sum_{i=1}^{n}\varphi\left(\frac{\langle e_i,u\rangle_+}{\lambda_0}\right)+\sum_{i=j}^{m-n}\varphi\left(\frac{\langle w_j,u\rangle_+}{\lambda_0}\right)= 1.
\end{align*}
Note that the convexity and strictly increasing  of $\varphi$ imply that, there exists a constant $\nu>0$ such that
 \begin{align*}
  \sum_{j=1}^{n}\varphi\left(\frac{\nu\langle -\mu'_je_j,u\rangle_+}{\lambda_0}\right)\geq\sum_{j=1}^{m-n}\varphi\left(\frac{\langle w_j,u\rangle_+}{\lambda_0}\right).
\end{align*}
We write $ \mu_j=\nu\mu'_j$, $j=1,\cdots ,n$, and define the set $\widetilde\Lambda=\{e_1,\cdots,e_n,-\mu_1e_1,\cdots,-\mu_ne_n\}$. Obviously, $\widetilde \Lambda$ is an obtuse set. Moreover, We have
\begin{align*}
\sum_{i=1}^{n}\varphi\left(\frac{\langle e_i,u\rangle_+}{\lambda_0}\right)+\sum_{j=1}^{n}\varphi\left(\frac{\langle -\mu_je_j,u\rangle_+}{\lambda_0}\right)\geq 1.
\end{align*}
Then by Lemma \ref{orl-nor-lem} we have $\lambda_0\leq h_{Z^{+}_\varphi \widetilde\Lambda}(u)$. Then we have
\begin{align}\label{asy-con-equ}
Z^+_\varphi\phi\Lambda\subseteq Z^+_\varphi\widetilde\Lambda=Z^+_\varphi\big\{e_1,\cdots,e_n,-\mu_1e_1,\cdots,-\mu_ne_n\big\}.
\end{align}

It remains to show that $Z^+_{\varphi} \phi\Lambda\subseteq Z^+_\varphi \Lambda_\bot$.
First note that $\Lambda$  is an obtuse,  $w_i$ has negative coordinates  with respect to the canonical basis $\Lambda_\bot$. In order to simply the computation,  we assume that $\Lambda'(\mu)=\{e_1,\cdots,e_n,-\mu e_1\}$, where $ \mu\geq 0$.
For $x\in e^\bot_1\cap e_2^\bot$, by (\ref {up-low-gra})  we have
\begin{align}
\label{inf-upp-fun}\overline g_{e_2}(Z^+_\varphi\Lambda'(\mu),x)&=\inf_{w\in e^\bot_2}\left\{h_{Z^+_\varphi\Lambda'(\mu)}(e_2+w)- \langle x, w\rangle\right\}.
 \end{align}
 Here
 \begin{align}\begin{split}\label{inf-orl-upp}
h_{Z^+_\varphi\Lambda'(\mu)}(e_2+w)&=\inf\bigg\{\lambda>0:\varphi\Big(\frac{\langle e_1,w\rangle_+}{\lambda}\Big)+\varphi\Big(\frac{\langle -\mu e_1,w\rangle_+}{\lambda}\Big)\\
&+\sum_{i=2}^n\varphi \Big(\frac{\langle e_i, e_2+w\rangle_+}{\lambda}\Big)\leq 1\bigg\},
 \end{split}\end{align}
Note that, for all $w\in e^\bot_2$, the scalar product $\langle x, w\rangle$ does not dependent on the first component of $w$. The increasing of $\varphi$ together with the expression of (\ref{inf-orl-upp}) show that  it suffices to compute the infimum of (\ref{inf-upp-fun}) over all $w\in e_1^\bot\cap e_2^\bot$. It is now obvious that the uppergraph function $\overline g_{e_2}(Z^+_\varphi\Lambda'(\mu),x)$ is independent of $\mu$ for every $x\in e^\bot\cap e^\bot_2$. The same argument applied to the lowergraph function leads to the same conclusion, so we infer that
 \begin{align*}
Z^+_\varphi\Lambda'(\mu)\cap e^\bot_1
 \end{align*}
 is independent of $\mu$. Moreover, the support function of $Z^+_\varphi \Lambda'(\mu)$ evaluated at vectors $w\in e^\bot_1$,
  \begin{align*}
    h_{Z^+_\varphi\Lambda'(\mu)}(w)=\inf\bigg\{\lambda>0:\sum_{i=2}^n\varphi\Big(\frac{\langle e_2,w\rangle_+}{\lambda}\Big)\leq 1\bigg\},
  \end{align*}
  is a constant function of $\mu$. Equivalently,
 \begin{align*}
Z^+_\varphi\Lambda'(\mu)|_{e^\bot_1},
 \end{align*}
 is independent of $\mu$.

If $\mu=1$, the convex body $Z^+_\varphi\Lambda'(1)$ is symmetric with respect to reflections in the hyperplane $e^\bot_1$. Then for $y\in Z^+_\varphi \Lambda'(\mu)$, we have
 \begin{align*}
y|_{e^\bot_1}\in Z^+_\varphi\Lambda'(\mu)|_{e^\bot_1}= Z^+_\varphi\Lambda'(1)|_{e^\bot_1}= Z^+_\varphi\Lambda'(1)\cap{e^\bot_1}= Z^+_\varphi\Lambda'(\mu)\cap{e^\bot_1}
 \end{align*}
for all $\mu$. In particular, $\underline g_{e_1}(Z^+_\varphi\Lambda'(\mu), y|_{e^\bot_1})$ is negative for all $\mu$. Moreover, the uppergraph function $\overline g_{e_1}(Z^+_\varphi\Lambda'(\mu),y|_{e_1^\bot})$ is independent of $\mu$. Because, $h_{Z^+_\varphi \Lambda'(\mu)}(e_1+w)$ is independent of $\mu$, for $w\in e^\bot_1$.
Hence
\begin{align}\begin{split}
y&\in\Big\{y|_{e^\bot_1}+re_1:0\leq r\leq \overline g_{e_1}(Z^+_\varphi \Lambda'(\mu),  y|_{e^\bot_1})\Big\}\\
  &=\Big\{y|_{e^\bot_1}+re_1:0\leq r\leq \overline g_{e_1}(Z^+_\varphi \Lambda'(0),  y|_{e^\bot_1})\Big\}\subset Z^+_\varphi\Lambda' (0).
\end{split}\end{align}
Then we have $Z^+_\varphi\Lambda' (\mu)\subseteq Z^+_\varphi \Lambda_\bot.$ This together with (\ref{asy-con-equ}) we have $Z^+_\varphi \Lambda'(\mu)= Z^+_\varphi\Lambda_\bot.$
Repeating this argument for $\mu_2,\cdots,\mu_n$, if them are not zero. We have that $\phi y$ is contained in $Z^+_\varphi \Lambda_\bot$, which shows the equality of (\ref{dis-orl-asy}).

Moreover, if we can show the intersection of any two distinct parts in the dissection (\ref{dis-orl-asy}) has volume zero, we can complete the proof. To see this, let $\Lambda^1,\Lambda^2\in \Lambda$, each contain $n$ vectors and assume that $\Lambda^1\neq\Lambda^2 .$ If one of these sets is not spanning, then the intersection $Z^+_\varphi\Lambda^1\cap Z^+_\varphi\Lambda^2$ is a set of volume zero contained in a hyperplane. Otherwise, without loss of generality, $\Lambda^1=\Lambda_\bot$ and $\Lambda^2$ does not contain $e_1$. Then by the definition of support function (\ref{asy-orl-zon}), we have $h_{Z^+_\varphi\Lambda^1}(-e_1)=0,$ and $h_{Z^+_\varphi\Lambda^2}(e_1)=0$. Then we obtain that $Z^+_\varphi\Lambda^1\cap Z^+_\varphi\Lambda^2$ is a set of volume zero contained in the hyperplane $e^\bot_1$.
So we complete the proof.
\end{proof}

If take $\varphi(t)=t^p$, $p\geq 1$, this result reduces to $L_p$ case.
\begin{cor}
   Suppose $p\geq 1$ and $\Lambda$ is a spanning obtuse set. Then
  \begin{align*}
    \frac{V(Z^+_p\Lambda)}{V(Z^+_1\Lambda)}=\frac{V(Z^+_p\Lambda_\bot)}{V(Z^+_1\Lambda_\bot)}.
  \end{align*}
\end{cor}

In paper \cite{cam-gro-on2006}, Campi and Gronchi proved that if $K_t$ is a shadow system of origin symmetric convex bodies in $\mathbb R^n$, then $V(K^*_t)^{-1}$ is a convex function of t. This result is developed by Meyer and Reisner \cite{mey-rei-sha2006} to more general setting.
\begin{prop}\cite{mey-rei-sha2006} \label{pol-sha-pro}
  Suppose $K_t$, $t\in[-a^{-1},1]$, is a shadow system of convex bodies along the direction $v=e_1$ and $V(K_t)$ is independent of $t$. Then the volume of $K_t^*$ is independent of $t$ if and only if there are a real number $\alpha$ and a vector $z\in \mathbb R^{1\times n-1}$ such that
\begin{align*}
  K_t=t\alpha e_1+\bigg(\begin{array}{cc}
                    1 &tz \\
                    0 & I_{n-1}
                  \end{array}\bigg)K_0.
\end{align*}
\end{prop}

Unfortunately, there is no analogue result for volume product of asymmetric Orlicz zonotopes with equality holds.
\begin{lem} \label{vol-pro-lem}
  Let $\varphi\in\mathcal C$, and $\Lambda=\Lambda_\bot\cup\{-\mu e_1\}$. Then
  \begin{align}\label{vol-orl-rat-lem}
    V(Z^{+,*}_\varphi\Lambda)V(Z^{+}_1\Lambda)\geq V(Z^{+,*}_\varphi\Lambda_\bot)V(Z^{+}_1\Lambda_\bot),
  \end{align}
with equality if and only if $\varphi=Id$.
\end{lem}
\begin{proof}
If $\varphi=Id $, that means $Z^{+}_\varphi\Lambda=Z^{+}_1\Lambda=Z_1\Lambda$. It is an immediate consequence of the fact that all parallelepipeds have the same volume product.

Now assume that $\varphi\neq Id$, let $\Lambda_t$, $t\in[-a^{-1},1]$, denote the orthogonalization of $\Lambda$ with respect to $e_1$ defined by (\ref{ort-sha-sys}).  Theorem (\ref{orl-vol-rat-thm}) shows that the inverse volume product of  asymmetric Orlicz zonotopes associate with $\Lambda_t$ is a convex function of $t$,  together with the convex function attains it's maxima at the boundary of compact intervals, we obtain
\begin{align*}
  \frac{1}{V(Z^{+,*}_\varphi\Lambda)V(Z^{+}_1\Lambda)}\leq \max_{t\in\{-a^{-1},1\}}\bigg\{\frac{1}{V(Z^{+,*}_\varphi\Lambda_t)V(Z^{+}_1\Lambda_t)}\bigg\}.
\end{align*}
By the $GL(n)$ invariance of the volume product of  asymmetric Orlicz zonotopes and the definition of $\Lambda_t$, the right hand side of this inequality is just
\begin{align*}
 \frac{1}{V(Z^{+,*}_\varphi\Lambda_\bot)V(Z^{+}_1\Lambda_\bot)}.
\end{align*}
Thus the equality condition of inequality (\ref{vol-orl-rat-lem}) means that the $V(Z_\varphi^{+,*}\Lambda_t)$ is a constant function of $t$. On the other hand,  by the definition of (\ref{ort-sha-sys}), we have
$$\Lambda_t=\big\{(1+ta)e_1,\mu(t-1)e_1,e_2,\cdots, e_n\big\},$$
here $\Lambda_t,$ $t\in[-a^{-1},1]$, is a spanning obtuse set. Together with Lemma \ref{orl-vol-lem} and the fact $Z^+_1\Lambda_t$ is independent of $t$, we obtain that $V(Z^+_\varphi\Lambda_t)$ is independent of $t$.  Proposition (\ref{pol-sha-pro}) implies that $Z^+_\varphi \Lambda_t$ are affine images of each other, which means there is a number $\alpha$ and a vector $z\in \mathbb R^{1\times(n-1)}$ such that
$$Z^+_\varphi\Lambda_t=t\alpha e_1+\phi_t Z^+_\varphi\Lambda,$$
where $\phi_t=\bigg(\begin{array}{cc}
         1 & tz \\
         0 & I_{n-1}
       \end{array}\bigg).$ Note that $Z^+_\varphi$ is $GL(n)$ equivariant, we can rewrite it as
\begin{align}\label{orl-aff}
Z^+_\varphi\Lambda_t=t\alpha e_1+ Z^+_\varphi\phi_t\Lambda.
\end{align}
Equivalent, for all $u\in\mathbb R^n$,
\begin{align}\label{orl-aff-equ}
  h_{Z^+_\varphi\Lambda_t}(u)=t\alpha\langle e_1,u\rangle+h_{Z^+_\varphi\phi_t\Lambda}(u).
\end{align}
Now we determined the constant $\alpha$. Note that $t\in[a^{-1},1]$, the zonotope $Z^+_\varphi\Lambda_t$ is symmetric with respect to permutations of all coordinates except the first. Due to (\ref{orl-aff}), this implies that $z$ has $n-1$ equal components, say $\xi$. Note that the coefficient $a$ is nothing to do with the $\xi$,  without loss of generality, we may assume  $\xi\leq 0$, let $u=e_1$ and $t=1$ in
(\ref{orl-aff-equ}), after a simple computation we obtain
\begin{align*}
  \alpha=\frac{a}{\varphi^{-1}(1)}=a.
\end{align*}
In order to determine $\xi$, firstly,  by the normalization of $\varphi$ and together with Lemma \ref{orl-nor-lem}, we have
\begin{align} \label{e2-int}
  h_{Z^+_\varphi \Lambda_1}(e_i)=1, \,\,\,\,\,\,\,\,\,\,\,\,i=2,\,\,\cdots, n.
\end{align}
Note that $Z^+_\varphi\Lambda_1$ is convex, specially, let $|e_2|=1=h_{Z^+_\varphi \Lambda_1}(e_2)$, which means that $e_2$ is contained in a plan intersect with $Z^+_\varphi\Lambda_1$, we say,
\begin{align}\label{e2-poi}
  \{e_2\}=Z^+_\varphi\Lambda_1\cap({e_2}+span\{e_1\}).
\end{align}
On other hand, by the definition of convex hull $conv$ and the support function of $Z^+_\varphi\Lambda$, together with  Lemma (\ref{orl-nor-lem}), we have $Z^+_\varphi\Lambda$ contains the convex hull of $\Lambda$,  that is
\begin{align*}
conv\{\Lambda\}\subseteq Z^+_\varphi\Lambda.
\end{align*}
Then we have $Z^+_\varphi \phi_1 \Lambda$ contains the convex hull of $\phi_1\Lambda$, $conv\{\phi_1\Lambda\}$. In particular, it contains $\phi_1 e_2=\xi e_1+e_2$. Combine this observation with (\ref{orl-aff}) and (\ref{e2-poi}) for $t=1$, we obtain
\begin{align*}
  e_2=(a+\xi)e_1+e_2,
\end{align*}
which means $\xi=-a$.

Now putting $u= e_1 +e_2$ and $t=-a^{-1}$ in equation (\ref{orl-aff-equ}).  Let $h_{Z^+_\varphi \Lambda_{-a^{-1}}}(e_1+e_2)=\lambda,\,\,
h_{Z^+_\varphi\phi_{-a^{-1}}\Lambda}(e_1+e_2)=\lambda'$. Note that $\Lambda_{-a^{-1}}=\{\mu(-a^{-1}-1)e_1,\,\,e_2,\cdots, e_n\}$  and $\phi_{-a^{-1}}\Lambda=\{e_1,\,\,e_1+e_2, \cdots, e_1+e_n, -\mu e_1\}$, then we have
\begin{align}\label{lam}
  \lambda=1,\,\,\,\,\,\,\,\,\,and\,\,\,\,\,\,\,\,\,\, (n-1)\varphi\Big(\frac{1}{\lambda'}\Big)+\varphi\Big(\frac{2}{\lambda'}\Big)=1.
\end{align}
Note that by equation (\ref{orl-aff-equ}), $\lambda$ and $\lambda'$ should satisfy $\lambda={-1}+\lambda'$. On other hand, by (\ref{lam}), they contradict with equality (\ref{orl-aff-equ}). Then $Z^+_\varphi\Lambda_t$ are not the affine images of each other. So the equality (\ref{vol-orl-rat-lem}) does not hold when $\varphi\neq Id.$ We complete the proof.
 \end{proof}
If take $\varphi(t)=t^p$, $p\geq 1$, it is established in \cite{web-sha2013}.
\begin{cor}
  Let $p\geq 1$, and $\Lambda=\Lambda_\bot\cup\{-\mu e_1\}$. Then
  \begin{align}
    V(Z^{+,*}_p\Lambda)V(Z^{+}_1\Lambda)\geq V(Z^{+,*}_p\Lambda_\bot)V(Z^{+}_1\Lambda_\bot),
  \end{align}
with equality if and only if $p=1$.
\end{cor}

\section{proofs of the main results}
Now we are in a position of proving the main results.
Before giving the main results, let us present the following lemma established by Weberndorfer, which we will use in the proof of our results.
\begin{lem}\cite{web-sha2013} \label{gen-vol-rat-lem}
  Suppose $\Phi$ is a real-valued $GL(n)$ invariant function on finite and spanning multisets. Moreover, assume that $\Phi(\Lambda_t)$ is a convex function of $t$ whenever $\Lambda_t$,  $t\in[-a^{-1},1]$ is an orthogonalization of a multiset $\Lambda$ defined by (\ref{ort-sha-sys}). Then for ever finite and spanning multiset $\Lambda$, there exists a multiset $\Lambda_{e_1}$ of multiples of $e_1$ such that
  \begin{align}\label{gen-vol-ine}
    \Phi(\Lambda)\leq \Phi(A_\bot \uplus \Lambda_{e_1}).
  \end{align}
  Moreover,

(\romannumeral 1) If $\Lambda$ is not a $GL(n)$ image of $\Lambda_\bot$ and equality holds in (\ref{gen-vol-ine}), then $\Lambda_{e_1}$ is not the empty set.

(\romannumeral 2) If $\Lambda$ is not a $GL(n)$ image of an obtuse set and equality holds in (\ref{gen-vol-ine}), then $\Lambda_{e_1}$ contains a positive multiple of $e_1$.
\end{lem}

\begin{thm}
  Let $\varphi\in\mathcal C$, and $\Lambda$ is a finite and spanning multiset. Then
  \begin{align}\label{the-one}
    V(Z^{+,*}_\varphi\Lambda)V(Z^+_1\Lambda)\geq V(Z^{+,*}_\varphi\Lambda_\bot)V(Z^+_1\Lambda_\bot).
  \end{align}
If $\varphi\neq Id$, with equality if and only if $\Lambda$ is a $GL(n)$ image of the canonical basis $\Lambda_\bot$. If $\varphi=Id$, the equality holds if and only if $Z^+_1\Lambda$ is a parallelepiped.
  \end{thm}
\begin{proof}
First, if $\varphi=Id$, it is established in \cite{web-sha2013}, so we only need to show the case $\varphi\neq Id$.

For $\varphi\neq Id$, let $\mathcal P(\Lambda)=\frac{1}{V(Z^{+,*}_\varphi\Lambda)V(Z^+_1\Lambda)}$ denote the inverse volume product of asymmetric Orlicz zonotopes, $\Lambda_t$, $t\in[-a^{-1},1]$, denotes the orthogonalization of $\Lambda$. Firstly, by the $GL(n)$ invariance of $\mathcal P(\Lambda)$, there is nothing to  show if $\Lambda$ is a $GL(n)$ image of the canonical basis $\Lambda_\bot$. Otherwise, by Theorem \ref{orl-vol-rat-thm}, we know that $\mathcal P(\Lambda_t)$, $t\in[-a^{-1},1]$, satisfies the hypotheses of Lemma \ref{gen-vol-rat-lem}. Then there exists a multiset $\Lambda_{e_1}$ of $e_1$ such that
  \begin{align}\label{vol-pro-ort}
\mathcal P(\Lambda)\leq \mathcal (\Lambda_\bot\uplus\Lambda_{e_1}).
  \end{align}
If $\Lambda_{e_1}$ is empty then the inequality (\ref{the-one}) holds.
If $\Lambda_{e_1}$ contains the only positive multiples of $e_1$, then $\overline{\Lambda_\bot\uplus \Lambda_{e_1} }$ is a $GL(n)$ image of $\Lambda_\bot$. Then we have
 \begin{align*}
\mathcal P(\Lambda)\leq \mathcal P (\Lambda_\bot\uplus\Lambda_{e_1})=\mathcal P (\Lambda_\bot).
  \end{align*}
It remains to show that if $\Lambda_{e_1}$ contains negative multiples, we say, $\Lambda_{e_1}=\{-\mu e_1\}$, where $\mu>0$. By Lemma \ref{vol-pro-lem} we have
\begin{align*}
\mathcal P(\Lambda)\leq \mathcal P (\Lambda_\bot\uplus\Lambda_{e_1})<\mathcal P (\Lambda_\bot).
  \end{align*}
 Then the inequality of (\ref{the-one}) holds. Now we deal with the equality condition. Since the equality holds in (\ref{vol-pro-ort}) only if $\Lambda_{e_1}$ is not empty. By Lemma \ref{spa-obt-lem}, we have
  \begin{align*}
    \mathcal P(\Lambda\uplus\Lambda_{e_1})\leq\mathcal P(\overline{\Lambda_\bot\uplus \Lambda_{e_1}}),
  \end{align*}
 with equality if and only if $\Lambda_{e_1}=\{-\mu{e_1}\}$, where $\mu\geq 0$. Note that $\Lambda_{e_1}$ is not the empty set, then $\mu>0$.
So we have
\begin{align} \label{equ-lam}
\mathcal P(\Lambda)=\mathcal P(\Lambda_\bot\uplus \Lambda_{e_1})=\mathcal P(\overline{\Lambda_\bot\cup\{-\mu e_1\}})=\mathcal P(\Lambda_\bot).
\end{align}
Note that $\varphi\neq Id$, by Lemma \ref{vol-pro-lem}, we have the equalities of (\ref{equ-lam}) hold if and only if $\Lambda$, $\Lambda_\bot\uplus \Lambda_{e_1}$, $\overline{\Lambda_\bot\cup\{-\mu e_1\}}$, and $\Lambda_\bot$ are $GL(n)$ images of each other. So we obtain the desired inequality together with its equality conditions.
\end{proof}

If $\varphi(t)=t^p$, $p\geq1$, this result reduces to the asymmetric $L_p$-volume ratio inequality in \cite{web-sha2013}.
\begin{cor}
 Let $p\geq 1$, and $\Lambda$ is a finite and spanning multiset. Then
  \begin{align*}
    V(Z^{+,*}_p\Lambda)V(Z^+_1\Lambda)\geq V(Z^{+,*}_p\Lambda_\bot)V(Z^+_1\Lambda_\bot).
  \end{align*}
If $p> 1$, with equality if and only if $\Lambda$ is a $GL(n)$ image of the canonical basis $\Lambda_\bot$. If $p=1$, the equality holds if and only if $Z^+_1\Lambda$ is a parallelepiped.
\end{cor}

\begin{thm}
  Suppose $\varphi\in\mathcal C$ and $\Lambda$ is a finite and spanning multiset. Then
  \begin{align}\label{vol-rat-ine}
    \frac{V(Z^+_\varphi\Lambda)}{V(Z^+_1\Lambda)}\leq\frac{V(Z^+_\varphi\Lambda_\bot)}{V(Z^+_1\Lambda_\bot)},
  \end{align}
  with equality if and only if $\Lambda$ is a $GL(n)$ image of an obtuse set.
\end{thm}
\begin{proof}
Let $\mathcal R(\Lambda)=\frac{V(Z^+_\varphi\Lambda)}{V(Z^+_1\Lambda)}$,
Theorem  \ref{orl-vol-rat-thm} implies that $\mathcal R(\Lambda)$ satisfies the hypotheses of Lemma \ref{gen-vol-rat-lem}. Then
\begin{align}\label{rad-orl}
  \mathcal R(\Lambda)\leq \mathcal R(\Lambda_\bot\uplus\Lambda_{e_1}),
\end{align}
where $\Lambda_{e_1}$ contains multiples of $e_1$. If $\Lambda_{e_1}$ is empty, then (\ref{vol-rat-ine}) holds. If $\Lambda_{e_1}$ contains the only positive multiples of $e_1$, then $\overline{\Lambda_\bot\uplus \Lambda_{e_1} }$ is a $GL(n)$ image of $\Lambda_\bot$. Then we have
 \begin{align}\label{b1}
 \mathcal R (\overline{\Lambda_\bot\uplus\Lambda_{e_1}})=\mathcal R (\Lambda_\bot).
  \end{align}
Moreover, by Lemma \ref{spa-obt-lem},
\begin{align}\label{b2}
 \mathcal R ({\Lambda_\bot\uplus\Lambda_{e_1}})\leq\mathcal R (\overline{\Lambda_\bot\uplus\Lambda_{e_1}}).
  \end{align}
Together with (\ref{rad-orl}), (\ref{b1}) and (\ref{b2}) we obtain
\begin{align*}
   \mathcal R(\Lambda)\leq\mathcal R (\Lambda_\bot).
\end{align*}
Note that if $\Lambda_{e_1}=\{-\mu e_1\}$, where $\mu>0$, then $\Lambda_\bot\uplus\Lambda_{e_1}=\Lambda_\bot\cup\{-\mu e_1\}$ is an obtuse set. Lemma \ref{orl-vol-lem} implies that the inequality in (\ref{vol-rat-ine}) holds.

Now we deal with the equality case of (\ref{vol-rat-ine}). We assume that $\Lambda$ is not a $GL(n)$ image of an obtuse set. Note that (\ref{rad-orl}) with equality holds only if $\Lambda_{e_1}$ contains a positive multiple of $e_1$. In this case, then $\overline{\Lambda_\bot\uplus \Lambda_{e_1} }$ is a $GL(n)$ image of $\Lambda_\bot$. Then we have
 \begin{align*}
 \mathcal R (\overline{\Lambda_\bot\uplus\Lambda_{e_1}})=\mathcal R (\Lambda_\bot).
  \end{align*}
By Lemma \ref{spa-obt-lem}, we have
\begin{align*}
 \mathcal R ({\Lambda_\bot\uplus\Lambda_{e_1}})\leq\mathcal R (\overline{\Lambda_\bot\uplus\Lambda_{e_1}}).
  \end{align*}
With equality if and only if ${\Lambda_\bot\uplus\Lambda_{e_1}}=\overline{\Lambda_\bot\uplus\Lambda_{e_1}}$, which means $\Lambda_{e_1}$ must be an negative multiples of $e_1$, then it contradicts with $\Lambda_{e_1}$ contains positive multiples of $e_1$. We prove that if the equality hold in (\ref{vol-rat-ine}), holds then $\Lambda$ is a $GL(n)$ image of an obtuse set.

On the other hand, if $\Lambda$ is a $GL(n)$ image of an obtuse set, by the $GL(n)$ invariance of the volume ratio for the Orlicz zonotopes and Lemma \ref{orl-vol-lem}, the equality of (\ref{vol-rat-ine}) holds.

Together with the above we have the equality of (\ref{vol-rat-ine}) hold if and only if $\Lambda$ is a $GL(n)$ image of an obtuse set. We complete the proof.
\end{proof}

If we take $\varphi(t)=t^p$, $p>1$, then it reduces to the following.
\begin{cor}
 Suppose $p>1$ and $\Lambda$ is a finite and spanning multiset. Then
  \begin{align*}
    \frac{V(Z^+_p\Lambda)}{V(Z^+_1\Lambda)}\leq\frac{V(Z^+_p\Lambda_\bot)}{V(Z^+_1\Lambda_\bot)},
  \end{align*}
  with equality if and only if $\Lambda$ is a $GL(n)$ image of an obtuse set.
\end{cor}

\bibliographystyle{amsplain}

\end{document}